\documentclass[12pt]{amsart}

\usepackage{amssymb,amsmath,graphicx,color,textcomp, amsthm,bbm,bbold, enumerate,booktabs}
\usepackage{chngcntr}
\usepackage{apptools}
\usepackage{mathrsfs}
\usepackage{tikz}
\usetikzlibrary{trees}

\AtAppendix{\counterwithin{theorem}{section}}
\definecolor{Red}{cmyk}{0,1,1,0}

\definecolor{verde}{cmyk}{1,0,1,0}

\definecolor{loka}{cmyk}{.5,0,1,.5}
\definecolor{azul}{cmyk}{1,1,0,0}

\numberwithin{equation}{section}

\newcommand{\be}{\begin{equation}}
\newcommand{\ee}{\end{equation}}

\newtheorem{theorem}{Theorem}

\newtheorem{definition}{Definition}

\newtheorem{remark}{Remark}
\newtheorem{corollary}[equation]{Corollary}

\newtheorem{example}{Example}

\begin{document}
\title{A truncated $\mathcal{V}$-fractional derivative in $\mathbb{R}^n$}
\author{J. Vanterler da C. Sousa$^1$}
\address{$^1$ Department of Applied Mathematics, Institute of Mathematics,
 Statistics and Scientific Computation, University of Campinas --
UNICAMP, rua S\'ergio Buarque de Holanda 651,
13083--859, Campinas SP, Brazil\newline
e-mail: {\itshape \texttt{ra160908@ime.unicamp.br, capelas@ime.unicamp.br }}}

\author{E. Capelas de Oliveira$^1$}

\begin{abstract}Using the six parameters truncated Mittag-Leffler function, we introduce a convenient truncated function to define the so-called truncated $\mathcal{V}$-fractional derivative type. After a discussion involving some properties associated with this derivative, we propose the derivative of a vector valued function and define the $\mathcal{V}$-fractional Jacobian matrix whose properties allow us to say that: the multivariable truncated $\mathcal{V}$-fractional derivative type, as proposed here, generalizes the truncated
$\mathcal{V}$-fractional derivative type and can bee extended to obtain a truncated $\mathcal{V}$-fractional partial derivative type. As applications we discuss and prove the change of order associated with two index i.e., the commutativity of two truncated $\mathcal{V}$-fractional partial derivative type and propose the truncated $\mathcal{V}$-fractional Green's theorem.

\end{abstract}
\maketitle

\section{Introduction}
Recently, Sousa and Oliveira \cite{JEC2} introduced the truncated $\mathcal{V}$-fractional derivative in the domain $\mathbb{R}$, satisfying classical properties of the integer-order calculus, having as special property, to unify five other formulations of local fractional derivatives of which we mention the derivatives: conformable fractional, alternative fractional, truncated alternative fractional, $M$-fractional and truncated $M$-fractional \cite{UNT2, KRHA, JEC, JEC1}.

In 2015, Atangana et al. \cite{AABD}, performed a work approaching new properties of the conformable fractional derivative, being the domain of the functions considered in $\mathbb{R}^n$. In 2017, Gözütok and Gözütok \cite {GNY} introduced the multivariable conformable fractional calculus, presenting interesting results found in $\mathbb{R}^n$. However, such a result is restricted only to the conformable fractional derivative. In this sense, we extend our definition of the truncated $\mathcal{V}$-fractional derivative to the $\mathbb{R}^n$ \cite{JEC2}, since such a derivative formulation unifies the remaining five.
We denote this new differential operator by $_{i}^{\rho }\mathbb{V}_{\gamma ,\beta ,\alpha }^{\delta ,p,q}(z)$, $z\in\mathbb{R}^n$, to differentiate from the operator $_{i}^{\rho }\mathcal{V}_{\gamma ,\beta ,\alpha }^{\delta ,p,q}(z)$, $z\in\mathbb{R}$, where the parameter $\alpha$, associated with the order of the derivative is such that $0<\alpha<1$, where $\gamma ,\beta ,\rho ,\delta \in \mathbb{C}$ and $p,q>0$ such that ${Re}\left( \gamma \right) >0$, ${Re}\left( \beta \right) >0$, ${Re}\left( \rho \right) >0$, ${Re}\left( \delta \right) >0$ and ${Re}\left( \gamma \right) +p\geq q$.

The article is organized as follows: in section 2, we present the truncated $\mathcal{V}$-fractional derivative by means of the truncated six parameters Mittag-Leffler function. Also, three theorems have been introduced that address linearity, product, divisibility, continuity, and the $\alpha$-differentiable chain rule. In section 3, we introduce our main result, the multivariable truncated $\mathcal{V}$-fractional derivative as well as results that justifies its continuity and uniqueness. In this sense, we introduce the $\mathcal{V}$-fractional Jacobian matrix and introduce and prove two theorems dealing with: chain rule, linearity and the product of functions through the $\alpha$-differentiable operator. In section 4, we present the concept of $\mathcal{V}$-fractional partial derivative and discuss two applications i.e., a theorem associated with the commutativity of two truncated $\mathcal{V}$-fractional derivatives and $\mathcal{V}$-fractional Green's theorem. Concluding remarks close the article.

\section{Preliminaries}

We will present the definition of the truncated $\mathcal{V}$-fractional derivative through the truncated six parameters Mittag-Leffler function and the gamma function. In this sense, we will present theorems that relate to the continuity and linearity, product, divisibility, as well as the chain rule.

Then, we begin with the definition of the six parameters truncated Mittag-Leffler function given by \cite{JEC2},
\begin{equation}\label{A9}
_{i}\mathbb{E}_{\gamma ,\beta ,p}^{\rho ,\delta ,q}\left( z\right) =\overset{i}{%
\underset{k=0}{\sum }}\frac{\left( \rho \right) _{qk}}{\left( \delta \right)
_{pk}}\frac{z^{k}}{\Gamma \left( \gamma k+\beta \right) },
\end{equation}
being $\gamma ,\beta ,\rho ,\delta \in \mathbb{C}$ and $p,q>0$ such that ${Re}\left( \gamma \right) >0$, ${Re}\left( \beta \right) >0$, ${Re}\left( \rho \right) >0$, ${Re}\left( \delta \right) >0$, ${Re}\left( \gamma \right) +p\geq q$ and $\left( \delta \right) _{pk}$, $\left( \rho \right) _{qk}$ given by
\begin{equation}\label{A6}
\left( \rho \right) _{qk}=\frac{\Gamma \left( \rho +qk\right) }{\Gamma\left( \rho \right) },
\end{equation}
a generalization of the Pochhammer symbol and $\Gamma(\cdot)$ is the function gamma.

From Eq.(\ref{A9}), we introduce the following truncated function, denoted by $_{i}H^{\rho,\delta,q}_{\gamma,\beta,p}(z)$, by means of
\begin{equation}\label{A10}
_{i}H_{\gamma ,\beta ,p}^{\rho ,\delta ,q}\left( z\right) :=\Gamma \left( \beta
\right) \;_{i}\mathbb{E}_{\gamma ,\beta ,p}^{\rho ,\delta ,q}\left( z\right) =\Gamma
\left( \beta \right) \overset{i}{\underset{k=0}{\sum }}\frac{\left( \rho
\right) _{kq}}{\left( \delta \right) _{kp}}\frac{z^{k}}{\Gamma \left( \gamma
k+\beta \right) }.
\end{equation}

In order to simplify notation, in this work, if the truncated $\mathcal{V}$-fractional derivative of order $\alpha$, according to Eq.(\ref{A11}) below, of a function $f$ exists, we simply say that the $f$ function is $\alpha$-differentiable.

So, we start with the following definition, which is a generalization of the usual definition of a derivative presented as a particular limit.

\begin{definition}\label{def7} Let $f:\left[ 0,\infty \right) \rightarrow \mathbb{R}$. For $0<\alpha <1$ the truncated $\mathcal{V}$-fractional derivative of $f$ of order $\alpha$, denoted by $_{i}^{\rho }\mathcal{V}_{\gamma ,\beta ,\alpha }^{\delta ,p,q}(\cdot)$, is defined as
\begin{equation}\label{A11}
_{i}^{\rho }\mathcal{V}_{\gamma ,\beta ,\alpha }^{\delta ,p,q}f\left( t\right) :=%
\underset{\epsilon \rightarrow 0}{\lim }\frac{f\left( t\;_{i}H_{\gamma ,\beta
,p}^{\rho ,\delta ,q}\left( \epsilon t^{-\alpha }\right) \right) -f\left(
t\right) }{\epsilon },
\end{equation}
for $\forall t>0$, $_{i}H_{\gamma ,\beta ,p}^{\rho ,\delta ,q}\left( \cdot\right) $ is a truncated function  as defined in {\rm Eq.(\ref{A10})} and being $\gamma ,\beta ,\rho ,\delta \in \mathbb{C}$ and $p,q>0$ such that ${Re}\left( \gamma \right) >0$, ${Re}\left( \beta \right) >0$, ${Re}\left( \rho \right) >0$, ${Re}\left( \delta \right) >0$, ${Re}\left( \gamma \right) +p\geq q$ and $\left( \delta \right) _{pk}$, $\left( \rho \right) _{qk}$ given by {\rm Eq.(\ref{A6})} {\rm \cite{JEC2}}. 
\end{definition}

Below, we recover three theorems (the proofs can be found in \cite{JEC2}) without proofs which are important in what follows.

\begin{theorem}\label{teo1} If the function $f:\left[ 0,\infty \right) \rightarrow \mathbb{R}$ is $\alpha $-differentiable for $t_{0}>0$, with $0<\alpha \leq 1$, then $f$ is continuous in $t_{0}$. 
\end{theorem} 

\begin{theorem}\label{teo2} Let $0<\alpha\leq 1$, $a,b\in\mathbb{R}$, $\gamma ,\beta ,\rho ,\delta \in \mathbb{C}$ and $p,q>0$ such that ${Re}\left( \gamma \right) >0$, ${Re}\left( \beta \right) >0$, ${Re}\left( \rho \right) >0$, ${Re}\left( \delta \right) >0$, ${Re}\left( \gamma \right) +p\geq q$ and $f,g$ $\alpha$-differentiable, for $t>0$. Then,
\begin{enumerate}
\item $_{i}^{\rho }\mathcal{V}_{\gamma ,\beta ,\alpha }^{\delta ,p,q}\left( af+bg\right)
\left( t\right) =a\, _{i}^{\rho }\mathcal{V}_{\gamma ,\beta ,\alpha }^{\delta
,p,q}f\left( t\right) +b\, _{i}^{\rho }\mathcal{V}_{\gamma ,\beta ,\alpha
}^{\delta ,p,q}g\left( t\right) $

\item $_{i}^{\rho }\mathcal{V}_{\gamma ,\beta ,\alpha }^{\delta ,p,q}\left( f\cdot g\right)
\left( t\right) =f\left( t\right) \, _{i}^{\rho }\mathcal{V}_{\gamma ,\beta
,\alpha }^{\delta ,p,q}g\left( t\right)+g\left( t\right) \,
_{i}^{\rho }\mathcal{V}_{\gamma ,\beta ,\alpha }^{\delta ,p,q}f\left( t\right) 
$

\item $_{i}^{\rho }\mathcal{V}_{\gamma ,\beta ,\alpha }^{\delta ,p,q}\left( \frac{f}{g}%
\right) \left( t\right) =\displaystyle\frac{g\left( t\right) \,_{i}^{\rho }\mathcal{V}_{\gamma
,\beta ,\alpha }^{\delta ,p,q}f\left( t\right) -f\left( t\right)
\, _{i}^{\rho }\mathcal{V}_{\gamma ,\beta ,\alpha }^{\delta ,p,q}g\left( t\right) }{\left[ g\left( t\right) \right] ^{2}}$

\item $_{i}^{\rho }\mathcal{V}_{\gamma ,\beta ,\alpha }^{\delta ,p,q}\left( c\right) =0$, where $f(t)=c$ is a constant.

\item If $f$ is differentiable, then $_{i}^{\rho }\mathcal{V}_{\gamma ,\beta ,\alpha }^{\delta ,p,q}f\left( t\right) =\displaystyle\frac{t^{1-\alpha }\Gamma \left( \beta \right) \left( \rho \right) _{q}}{\Gamma\left( \gamma +\beta \right) \left( \delta \right) _{p}}\frac{df\left(
t\right) }{dt}$.

\item $_{i}^{\rho }\mathcal{V}_{\gamma ,\beta ,\alpha }^{\delta ,p,q}\left(
t^{a}\right) =\displaystyle\frac{\Gamma \left( \beta \right) \left( \rho \right) _{q}}{%
\Gamma \left( \gamma +\beta \right) \left( \delta \right) _{p}}at^{a-\alpha }.$
\end{enumerate}
\end{theorem}

\begin{theorem}\label{teo3} {\rm\text{(Chain rule)}} Assume $f,g:(0,\infty)\rightarrow \mathbb{R}$  be two $\alpha$-differentiable functions where $0<\alpha\leq 1$. Let $\gamma ,\beta ,\rho ,\delta \in \mathbb{C}$ and $p,q>0$ such that ${Re}\left( \gamma \right) >0$, ${Re}\left( \beta \right) >0$, ${Re}\left( \rho \right) >0$, ${Re}\left( \delta \right) >0$, ${Re}\left( \gamma \right) +p\geq q$ then $\left( f\circ g\right)$ is $\alpha$-differentiable and for all $t>0$, we have
\begin{equation*}
_{i}^{\rho }\mathcal{V}_{\gamma ,\beta ,\alpha }^{\delta ,p,q}\left( f\circ
g\right) \left( t\right) =f^{\prime }\left( g\left( t\right) \right) \,
_{i}^{\rho }\mathcal{V}_{\gamma ,\beta ,\alpha }^{\delta ,p,q}g\left(
t\right),
\end{equation*}
for $f$ differentiable in $g(t)$.
\end{theorem}

\begin{definition}\label{def9} {\rm($\mathcal{V}$-fractional integral)} Let $a\geq 0$ and $t\geq a$. Also, let $f$ be a function defined on $(a,t]$ and $0<\alpha<1$. Then, the $\mathcal{V}$-fractional integral of $f$ of order $\alpha$ is defined by
\begin{equation}\label{A15}
_{a}^{\rho }\mathcal{I}_{\gamma ,\beta ,\alpha }^{\delta ,p,q}f\left( t\right) :=\frac{\Gamma \left( \gamma +\beta \right) \left( \delta \right) _{p}}{\Gamma\left( \beta \right) \left( \rho \right) _{q}}\int_{a}^{t}\frac{f\left(x\right) }{x^{1-\alpha }}dx,
\end{equation}
with $\gamma ,\beta ,\rho ,\delta \in \mathbb{C}$ and $p,q>0$ such that ${Re}\left( \gamma \right) >0$, ${Re}\left( \beta \right) >0$, ${Re}\left( \rho \right) >0$, ${Re}\left( \delta \right) >0$ and ${Re}\left( \gamma \right) +p\geq q$.
\end{definition}

\begin{remark}\label{fe} In order to simplify notation, in this work, the $\mathcal{V}$-fractional integral of order $\alpha$, will be denoted by
\begin{equation*}
\frac{\Gamma \left( \gamma +\beta \right) \left( \delta \right) _{p}}{\Gamma
\left( \beta \right) \left( \rho \right) _{q}}\int_{a}^{b}\frac{f\left(
t\right) }{t^{1-\alpha }}dt=\int_{a}^{b}f\left( t\right) d_{\omega }t
\end{equation*}
where, $d_{\omega }t=\displaystyle\frac{\Gamma \left( \gamma +\beta \right) \left( \delta \right)_{p}}{\Gamma \left( \beta \right) \left( \rho \right) _{q}}t^{\alpha -1}dt$.
\end{remark}

\section{$\mathcal{V}$-fractional derivative of a vector valued function}

In this section, we present our main result, the truncated $\mathcal{V}$-fractional derivative in $\mathbb{R}^n$ and check its continuity as well as the uniqueness of linear transformation. We present the definition of the truncated $\mathcal{V}$-fractional Jacobian matrix, the chain rule and the theorem that refers to linearity and product. We conclude the section discussing some examples.

\begin{definition}\label{AL} Let $f$ be a vector valued function with $n$ real variables such that $f(x_{1}, x_{2},...,x_{n})=(f_{1}(x_{1}, x_{2},...,x_{n}), f_{2}(x_{1}, x_{2},...,x_{n}),...,f_{m}(x_{1}, x_{2},...,x_{n}))$. We say that $f$ is  $\alpha$-differentiable at $a=(a_{1},...,a_{n})\in\mathbb{R}^n$ where each $a_{i}>0$, if there is a linear transformation $L:\mathbb{R}^n \rightarrow \mathbb{R}^m$ such that
\begin{equation}
\underset{\varepsilon \rightarrow 0}{\lim }\frac{\left\Vert f\left(
a_{1}\,_{i}H_{\gamma ,\beta ,p}^{\rho ,\delta ,q}\left( \varepsilon
_{1}a_{1}^{-\alpha }\right) ,...,a_{n}\,_{i}H_{\gamma ,\beta ,p}^{\rho ,\delta
,q}\left( \varepsilon _{n}a_{n}^{-\alpha }\right) \right) -f\left(
a_{1},...,a_{n}\right) -L\left( \varepsilon \right) \right\Vert }{\left\Vert
\varepsilon \right\Vert }=0,
\end{equation}
where $\varepsilon =\left( \varepsilon _{1},...,\varepsilon _{n}\right) $, $0<\alpha \leq 1$, $_{i}H_{\gamma ,\beta ,p}^{\rho ,\delta ,q}\left( \cdot \right) $ is the truncated function and $\rho,\delta,\gamma,\beta\in\mathbb{C}$, $p,q>0$ with, $Re(\rho)>0$, $Re(\delta)>0$, $Re(\gamma)>0$, $Re(\beta)>0$ and $Re(\gamma)+p\geq q$. The linear transformation is denoted by $_{i}^{\rho }\mathbb{V}^{\delta,p,q}_{\gamma,\beta,\alpha}f(a)$ and called the multivariable truncated $\mathcal{V}$-fractional derivative of $f$ of order $\alpha$ at $a$. 
\end{definition}

\begin{remark} Taking $m=n=1$ in {\rm Definition \ref{AL}}, we have 
\begin{equation}\label{AL1}
L\left( \varepsilon \right) =f\left( a\,_{i}H_{\gamma ,\beta ,p}^{\rho ,\delta
,q}\left( \varepsilon a^{-\alpha }\right) \right) -f\left( a\right) -r\left(
\varepsilon \right).
\end{equation}

Dividing by $\varepsilon$ both sides of {\rm Eq. \ref{AL1}} and taking the limit $\varepsilon \rightarrow 0$, we have
\begin{eqnarray}
\underset{\varepsilon \rightarrow 0}{\lim }\frac{L\left( \varepsilon \right) 
}{\varepsilon } &=&\underset{\varepsilon \rightarrow 0}{\lim }\frac{f\left(
a\,_{i}H_{\gamma ,\beta ,p}^{\rho ,\delta ,q}\left( \varepsilon a^{-\alpha
}\right) \right) -f\left( a\right) -r\left( \varepsilon \right) }{%
\varepsilon }  \notag \\
&=&\underset{\varepsilon \rightarrow 0}{\lim }\frac{f\left( a\,_{i}H_{\gamma
,\beta ,p}^{\rho ,\delta ,q}\left( \varepsilon a^{-\alpha }\right) \right)
-f\left( a\right) }{\varepsilon }  \notag \\
&=&\, _{i}^{\rho }\mathcal{V}_{\gamma ,\beta ,\alpha }^{\delta ,p,q}f\left(
a\right), 
\end{eqnarray}
where $\underset{\varepsilon \rightarrow 0}{\lim }\displaystyle\frac{r\left( \varepsilon \right) 
}{\varepsilon }=0$.
Thus, we conclude that, {\rm Definition \ref{AL}} is equivalent to {\rm Definition \ref{def7}}.
\end{remark}

\begin{theorem} Let $f$ be a vector valued function with $n$ variables. If $f$ is $\alpha$-differentiable at $a=(a_{1},...,a_{n})\in\mathbb{R}^n$ with $a_{i}>0$, then there is a unique linear transformations $L: \mathbb{R}^n \rightarrow \mathbb{R}^m$ such that
\begin{equation}
\underset{\varepsilon \rightarrow 0}{\lim }\frac{\left\Vert f\left(
a_{1}\,_{i}H_{\gamma ,\beta ,p}^{\rho ,\delta ,q}\left( \varepsilon
_{1}a_{1}^{-\alpha }\right) ,...,a_{n}\,_{i}H_{\gamma ,\beta ,p}^{\rho ,\delta
,q}\left( \varepsilon _{n}a_{n}^{-\alpha }\right) \right) -f\left(
a_{1},...,a_{n}\right) -L\left( \varepsilon \right) \right\Vert }{\left\Vert
\varepsilon \right\Vert }=0,
\end{equation}
with $0<\alpha \leq 1$, $_{i}H_{\gamma ,\beta ,p}^{\rho ,\delta ,q}\left( \cdot \right) $ is the truncated function and $\rho,\delta,\gamma,\beta\in\mathbb{C}$, $p,q>0$ such that, $Re(\rho)>0$, $Re(\delta)>0$, $Re(\gamma)>0$, $Re(\beta)>0$ and $Re(\gamma)+p\geq q$.
\end{theorem}

\begin{proof} Let $M:\mathbb{R}^n \rightarrow \mathbb{R}^m$ such that
\begin{equation*}
\underset{\varepsilon \rightarrow 0}{\lim }\frac{\left\Vert f\left(
a_{1}\,_{i}H_{\gamma ,\beta ,p}^{\rho ,\delta ,q}\left( \varepsilon
_{1}a_{1}^{-\alpha }\right) ,...,a_{n}\,_{i}H_{\gamma ,\beta ,p}^{\rho ,\delta
,q}\left( \varepsilon _{n}a_{n}^{-\alpha }\right) \right) -f\left(
a_{1},...,a_{n}\right) -M\left( \varepsilon \right) \right\Vert }{\left\Vert
\varepsilon \right\Vert }=0.
\end{equation*}

Hence,
\begin{eqnarray*}
&&\underset{\varepsilon \rightarrow 0}{\lim }\frac{\left\Vert L(\varepsilon
)-M(\varepsilon )\right\Vert }{\left\Vert \varepsilon \right\Vert }  \notag
\\
&\leq &\underset{\varepsilon \rightarrow 0}{\lim }\frac{\left\Vert
L(\varepsilon )-f\left( a_{1}\,_{i}H_{\gamma ,\beta ,p}^{\rho ,\delta
,q}\left( \varepsilon _{1}a_{1}^{-\alpha }\right) ,...,a_{n}\,_{i}H_{\gamma
,\beta ,p}^{\rho ,\delta ,q}\left( \varepsilon _{n}a_{n}^{-\alpha }\right)
\right) +f(a)\right\Vert }{\left\Vert \varepsilon \right\Vert }  \notag \\
&+&\underset{\varepsilon \rightarrow 0}{\lim }\frac{\left\Vert f\left(
a_{1}\,_{i}H_{\gamma ,\beta ,p}^{\rho ,\delta ,q}\left( \varepsilon
_{1}a_{1}^{-\alpha }\right) ,...,a_{n}\,_{i}H_{\gamma ,\beta ,p}^{\rho
,\delta ,q}\left( \varepsilon _{n}a_{n}^{-\alpha }\right) \right)
-f(a)-M\left( \varepsilon \right) \right\Vert }{\left\Vert \varepsilon
\right\Vert }=0,
\end{eqnarray*}
then
\begin{eqnarray*}
\underset{\varepsilon \rightarrow 0}{\lim }\frac{\left\Vert L(\varepsilon)-M(\varepsilon) \right\Vert }
{\left\Vert\varepsilon \right\Vert } &\leq& 0.
\end{eqnarray*}
If $x\in\mathbb{R}^n$, then $\varepsilon x\rightarrow 0$ as $\varepsilon\rightarrow 0$. Hence, for $x\neq0$ we have
\begin{eqnarray*}
0 =\underset{\varepsilon \rightarrow 0}{\lim }\frac{\left\Vert L(\varepsilon x)-M(\varepsilon x) \right\Vert }
{\left\Vert\varepsilon x \right\Vert }=\frac{\left\Vert L(x)-M(x) \right\Vert }{\left\Vert x \right\Vert }.
\end{eqnarray*}
Therefore $L(x)=M(x)$. We conclude that, $L$ is unique.
\end{proof}

\begin{example} Let us consider the function $f$ defined by $f(x,y)=sin(x)$ and the point $(a,b)\in\mathbb{R}^2$ such that $a,b>0$, then $_{i}^{\rho }\mathbb{V}^{\delta,p,q}_{\gamma,\beta,\alpha}f(a,b)=L$ satisfies $L(x,y)=\displaystyle\frac{\Gamma(\beta) (\rho)_{q}}{\Gamma(\gamma+\beta)(\delta)_{p}} x a^{1-\alpha} cos(a)$.

To prove this, we note that
\begin{eqnarray*}
&&\underset{\left( \varepsilon _{1},\varepsilon _{2}\right) \rightarrow
\left( 0,0\right) }{\lim }\displaystyle\frac{\left\vert f\left( a\,_{i}H_{\gamma ,\beta
,p}^{\rho ,\delta ,q}\left( \varepsilon _{1}a^{-\alpha }\right)
,b\,_{i}H_{\gamma ,\beta ,p}^{\rho ,\delta ,q}\left( \varepsilon
_{2}b^{-\alpha }\right) \right) -f\left( a,b\right) -L\left( \varepsilon
_{1},\varepsilon _{2}\right) \right\vert }{\left\Vert \left( \varepsilon
_{1},\varepsilon _{2}\right) \right\Vert }  \notag \\
&=&\underset{\left( \varepsilon _{1},\varepsilon _{2}\right) \rightarrow
\left( 0,0\right) }{\lim }\displaystyle\frac{\left\vert \sin \left( a\,_{i}H_{\gamma
,\beta ,p}^{\rho ,\delta ,q}\left( \varepsilon _{1}a^{-\alpha }\right)
\right) -\sin \left( a\right) -L\left( \varepsilon _{1},\varepsilon
_{2}\right) \right\vert }{\sqrt{\varepsilon _{1}^{2}+\varepsilon _{2}^{2}}} 
\notag \\
&\leq &\underset{\varepsilon _{1}\rightarrow 0}{\lim }\frac{\left\vert \sin
\left( a\,_{i}H_{\gamma ,\beta ,p}^{\rho ,\delta ,q}\left( \varepsilon
_{1}a^{-\alpha }\right) \right) -\sin \left( a\right) -\frac{\Gamma
\left( \beta \right) \left( \rho \right) _{q}}{\Gamma \left( \gamma +\beta
\right) \left( \delta \right) _{p}}\varepsilon _{1}a^{1-\alpha }\cos \left(
a\right) \right\vert }{\left\vert \varepsilon _{1}\right\vert }=0
\end{eqnarray*}
\end{example}

\begin{example} Let us consider the function $f$ defined by $f(x,y)=e^x$ and the point $(a,b)\in \mathbb{R}^2$ such that $a,b>0$, then $_{i}^{\rho }\mathbb{V}^{\delta,p,q}_{\gamma,\beta,\alpha}f(a,b)=L$ satisfies $L(x,y)=\displaystyle\frac{\Gamma(\beta) (\rho)_{q}}{\Gamma(\gamma+\beta)(\delta)_{p}} x a^{1-\alpha} e^a$.

To prove this, we note that
\begin{eqnarray*}
&&\underset{\left( \varepsilon _{1},\varepsilon _{2}\right) \rightarrow
\left( 0,0\right) }{\lim }\frac{\left\vert f\left( a\,_{i}H_{\gamma ,\beta
,p}^{\rho ,\delta ,q}\left( \varepsilon _{1}a^{-\alpha }\right)
,b\,_{i}H_{\gamma ,\beta ,p}^{\rho ,\delta ,q}\left( \varepsilon
_{2}b^{-\alpha }\right) \right) -f\left( a,b\right) -L\left( \varepsilon
_{1},\varepsilon _{2}\right) \right\vert }{\left\Vert \left( \varepsilon
_{1},\varepsilon _{2}\right) \right\Vert }  \notag \\
&=&\underset{\left( \varepsilon _{1},\varepsilon _{2}\right) \rightarrow
\left( 0,0\right) }{\lim }\frac{\left\vert e^{a\,_{i}H_{\gamma ,\beta
,p}^{\rho ,\delta ,q}\left( \varepsilon _{1}a^{-\alpha }\right)
}-e^{a}-L\left( \varepsilon _{1},\varepsilon _{2}\right) \right\vert }{\sqrt{%
\varepsilon _{1}^{2}+\varepsilon _{2}^{2}}}  \notag \\
&\leq &\left\vert \underset{\varepsilon _{1}\rightarrow 0}{\lim }\frac{%
e^{a\,_{i}H_{\gamma ,\beta ,p}^{\rho ,\delta ,q}\left( \varepsilon
_{1}a^{-\alpha }\right) }-e^{a}}{\varepsilon _{1}}-\frac{\Gamma \left(
\beta \right) \left( \rho \right) _{q}}{\Gamma \left( \gamma +\beta \right)
\left( \delta \right) _{p}}a^{1-\alpha }e^{a}\right\vert =0.
\end{eqnarray*}
\end{example}

\begin{definition} Consider the matrix of the linear transformation $_{i}^{\rho }\mathbb{V}^{\delta,p,q}_{\gamma,\beta,\alpha}f(a): \mathbb{R}^n \rightarrow \mathbb{R}^m$ with respect to the usual base of $\mathbb{R}^n$ and $\mathbb{R}^m$. This $m \times n$ matrix is called the truncated $\mathcal{V}$-fractional Jacobian matrix of $f$ at $a$, and denoted by $^{\rho }J_{\gamma ,\beta ,\alpha }^{\delta ,p,q}f\left( a\right) $, where $\rho,\delta,\gamma,\beta\in\mathbb{C}$, $p,q>0$ with, $Re(\rho)>0$, $Re(\delta)>0$, $Re(\gamma)>0$, $Re(\beta)>0$ and $Re(\gamma)+p\geq q$.
\end{definition}

\begin{example} If $f(x,y)=\sin (x)$, then we have the matrix
\begin{equation*}
^{\rho }J_{\gamma ,\beta ,\alpha }^{\delta ,p,q}f\left( a,b\right) =\left[ 
\begin{array}{ccc}
\displaystyle\frac{\Gamma \left( \beta \right) \left( \rho \right) _{q}}{\Gamma \left(
\gamma +\beta \right) \left( \delta \right) _{p}}a^{1-\alpha }\cos \left(
a\right)  &  & 0%
\end{array}
\right].
\end{equation*}
\end{example}

\begin{theorem} If a vector function $f$ with $n$ variables is $\alpha$-differentiable at $a=(a_{1},a_{2},...,a_{n})\in\mathbb{R}^n$, with $a_{i}>0$, then $f$ is continuous at $a\in\mathbb{R}^n$.
\end{theorem}
\begin{proof} Note that, 
\begin{eqnarray}
&&\left\Vert f\left( a_{1}\,_{i}H_{\gamma ,\beta ,p}^{\rho ,\delta ,q}\left(
\varepsilon _{1}a_{1}^{-\alpha }\right) ,...,a_{n}\,_{i}H_{\gamma ,\beta
,p}^{\rho ,\delta ,q}\left( \varepsilon _{n}a_{n}^{-\alpha }\right) \right)
-f(a_{1},...,a_{n})\right\Vert   \notag  \label{J1} \\
&\leq &\frac{\left\Vert f\left( a_{1}\,_{i}H_{\gamma ,\beta ,p}^{\rho
,\delta ,q}\left( \varepsilon _{1}a_{1}^{-\alpha }\right)
,...,a_{n}\,_{i}H_{\gamma ,\beta ,p}^{\rho ,\delta ,q}\left( \varepsilon
_{n}a_{n}^{-\alpha }\right) \right) -f(a_{1},...,a_{n})-L\left( \varepsilon
\right) \right\Vert \left\Vert \varepsilon \right\Vert }{\left\Vert
\varepsilon \right\Vert }  \notag \\
&&+\left\Vert L\left( \varepsilon \right) \right\Vert .
\end{eqnarray}

Taking the limit $\varepsilon \rightarrow 0$ in both sides of the {\rm Eq.(\ref{J1})}, we have
\begin{eqnarray*}
&&\underset{\varepsilon \rightarrow 0}{\lim }\left\Vert f\left(
a_{1}\,_{i}H_{\gamma ,\beta ,p}^{\rho ,\delta ,q}\left( \varepsilon
_{1}a_{1}^{-\alpha }\right) ,...,a_{n}\,_{i}H_{\gamma ,\beta ,p}^{\rho
,\delta ,q}\left( \varepsilon _{n}a_{n}^{-\alpha }\right) \right)
-f(a_{1},...,a_{n})\right\Vert  \\
&\leq &\underset{\varepsilon \rightarrow 0}{\lim }\frac{\left\Vert f\left(
a_{1}\,_{i}H_{\gamma ,\beta ,p}^{\rho ,\delta ,q}\left( \varepsilon
_{1}a_{1}^{-\alpha }\right) ,...,a_{n}\,_{i}H_{\gamma ,\beta ,p}^{\rho
,\delta ,q}\left( \varepsilon _{n}a_{n}^{-\alpha }\right) \right)
-f(a_{1},...,a_{n})-L\left( \varepsilon \right) \right\Vert }{\left\Vert
\varepsilon \right\Vert } \\
&&\times \underset{\varepsilon \rightarrow 0}{\lim }\left\Vert \varepsilon
\right\Vert +\underset{\varepsilon \rightarrow 0}{\lim }\left\Vert L\left(
\varepsilon \right) \right\Vert .
\end{eqnarray*}

Let $(u_{1},...,u_{n})=(\varepsilon_{1}a_{1}^{-\alpha},...,\varepsilon_{n}a_{n}^{-\alpha})$, then $u\rightarrow 0$ as $\varepsilon\rightarrow 0$. Since
\begin{equation*}
\underset{\varepsilon \rightarrow 0}{\lim }\left\Vert f\left(
a\,_{i}H_{\gamma ,\beta ,p}^{\rho ,\delta ,q}\left( u\right) \right)
-f(a)\right\Vert \leq 0,
\end{equation*}
we have,
\begin{equation*}
\underset{\varepsilon \rightarrow 0}{\lim }\left\Vert f\left(
a\,_{i}H_{\gamma ,\beta ,p}^{\rho ,\delta ,q}\left( u\right) \right)
-f(a)\right\Vert =0.
\end{equation*}

Hence, $f$ is continuous at $a\in\mathbb{R}^n$.
\end{proof}

\begin{theorem}\label{AR} {\rm\text{(Chain rule)}} Let $x\in\mathbb{R}^n$, $y\in\mathbb{R}^m$. If $f(x)=(f_{1}(x),...,f_{m}(x))$ is $\alpha$-differentiable at $a=(a_{1},...,a_{n})\in\mathbb{R}^n$, with $a_{i}>0$ such that $\alpha\in (0,1]$, and $g(y)=(g_{1}(y),...,g_{p}(y))$ is $\alpha$-differentiable at $f(a)\in\mathbb{R}^m$, with $f_{i}(a)>0$ such that $\alpha\in (0,1]$, then the composition $g\circ f$ is $\alpha$-differentiable at $a$ and
\begin{equation}
_{i}^{\rho }\mathbb{V}_{\gamma ,\beta ,\alpha }^{\delta ,p,q}\left( g\circ
f\right) \left( a\right) =g^{\prime }\left( f\left( a\right) \right)
\,_{i}^{\rho }\mathbb{V}_{\gamma ,\beta ,\alpha }^{\delta ,p,q}f\left(
a\right) ,
\end{equation}
for $g$ differentiable in $f(a)$ and $\rho,\delta,\gamma,\beta\in\mathbb{C}$, $p,q>0$ such that, $Re(\rho)>0$, $Re(\delta)>0$, $Re(\gamma)>0$, $Re(\beta)>0$ and $Re(\gamma)+p\geq q$.
\end{theorem}
\begin{proof}
Taking $ L=\,_{i}^{\rho }\mathbb{V}_{\gamma ,\beta ,\alpha }^{\delta,p,q}f\left( t\right) $ and $M=D g\left( f\left( a\right) \right) $, where $D$ is the derivative operator of integer order, we define,
\begin{eqnarray}\label{X1}
&&\varphi \left( a_{1}\,_{i}H_{\gamma ,\beta ,p}^{\rho ,\delta ,q}\left(
\varepsilon _{1}a_{1}^{-\alpha }\right) ,...,a_{n}\,_{i}H_{\gamma ,\beta
,p}^{\rho ,\delta ,q}\left( \varepsilon _{n}a_{n}^{-\alpha }\right) \right) 
\notag \\
&=&f\left( a_{1}\,_{i}H_{\gamma ,\beta ,p}^{\rho ,\delta ,q}\left(
\varepsilon _{1}a_{1}^{-\alpha }\right) ,...,a_{n}\,_{i}H_{\gamma ,\beta
,p}^{\rho ,\delta ,q}\left( \varepsilon _{n}a_{n}^{-\alpha }\right) \right)
-f(a)-L\left( \varepsilon \right),
\end{eqnarray}
\begin{eqnarray}\label{X2}
&&\psi \left( f_{1}\left( a\right) \,_{i}H_{\gamma ,\beta ,p}^{\rho ,\delta
,q}\left( k_{1}f_{1}\left( a\right) ^{-\alpha }\right) ,...,f_{n}\left(
a\right) \,_{i}H_{\gamma ,\beta ,p}^{\rho ,\delta ,q}\left( k_{n}f_{n}\left(
a\right) ^{-\alpha }\right) \right)   \notag \\
&=&g\left( f_{1}\left( a\right) \,_{i}H_{\gamma ,\beta ,p}^{\rho ,\delta
,q}\left( k_{1}f_{1}\left( a\right) ^{-\alpha }\right) ,...,f_{n}\left(
a\right) \,_{i}H_{\gamma ,\beta ,p}^{\rho ,\delta ,q}\left( k_{n}f_{n}\left(
a\right) ^{-\alpha }\right) \right) -g\left( f(a\right) )-M\left( k\right)\notag \\
\end{eqnarray}
and
\begin{eqnarray}\label{X3}
&&\rho \left( a_{1}\,_{i}H_{\gamma ,\beta ,p}^{\rho ,\delta ,q}\left(
\varepsilon _{1}a_{1}^{-\alpha }\right) ,...,a_{n}\,_{i}H_{\gamma ,\beta
,p}^{\rho ,\delta ,q}\left( \varepsilon _{n}a_{n}^{-\alpha }\right) \right)
=g\circ f\left( a_{1}\,_{i}H_{\gamma ,\beta ,p}^{\rho ,\delta ,q}\left(
\varepsilon _{1}a_{1}^{-\alpha }\right) ,...,a_{n}\,_{i}H_{\gamma ,\beta
,p}^{\rho ,\delta ,q}\left( \varepsilon _{n}a_{n}^{-\alpha }\right) \right) \notag
\\
&&-g\circ f(a)-M\circ L\left( \varepsilon \right). 
\end{eqnarray}

Hence, taking $\varepsilon \rightarrow 0$ and $k\rightarrow 0$ in both sides of $\rm {Eq.(\ref{X1})}$ and $\rm {Eq.(\ref{X2})}$, we have
\begin{eqnarray}\label{X4}
&&\underset{\varepsilon \rightarrow 0}{\lim }\frac{\left\Vert \varphi \left(
a_{1}\,_{i}H_{\gamma ,\beta ,p}^{\rho ,\delta ,q}\left( \varepsilon
_{1}a_{1}^{-\alpha }\right) ,...,a_{n}\,_{i}H_{\gamma ,\beta ,p}^{\rho
,\delta ,q}\left( \varepsilon _{n}a_{n}^{-\alpha }\right) \right)
\right\Vert }{\left\Vert \varepsilon \right\Vert }  \notag \\
&=&\lim \frac{\left\Vert f\left( a_{1}\,_{i}H_{\gamma ,\beta ,p}^{\rho
,\delta ,q}\left( \varepsilon _{1}a_{1}^{-\alpha }\right)
,...,a_{n}\,_{i}H_{\gamma ,\beta ,p}^{\rho ,\delta ,q}\left( \varepsilon
_{n}a_{n}^{-\alpha }\right) \right) -f(a)-L\left( \varepsilon \right)
\right\Vert }{\left\Vert \varepsilon \right\Vert }=0\notag \\
\end{eqnarray}
and
\begin{eqnarray}\label{X5}
&&\underset{k\rightarrow 0}{\lim }\frac{\left\Vert \psi \left( f_{1}\left(
a\right) \,_{i}H_{\gamma ,\beta ,p}^{\rho ,\delta ,q}\left( k_{1}f_{1}\left(
a\right) ^{-\alpha }\right) ,...,f_{n}\left( a\right) \,_{i}H_{\gamma ,\beta
,p}^{\rho ,\delta ,q}\left( k_{n}f_{n}\left( a\right) ^{-\alpha }\right)
\right) \right\Vert }{\left\Vert k\right\Vert }  \notag \\
&=&\underset{k\rightarrow 0}{\lim }\frac{\left\Vert g\left( f_{1}\left(
a\right) \,_{i}H_{\gamma ,\beta ,p}^{\rho ,\delta ,q}\left( k_{1}f_{1}\left(
a\right) ^{-\alpha }\right) ,...,f_{n}\left( a\right) \,_{i}H_{\gamma ,\beta
,p}^{\rho ,\delta ,q}\left( k_{n}f_{n}\left( a\right) ^{-\alpha }\right)
\right) -g\left( f(a\right) )-M\left( k\right) \right\Vert }{\left\Vert
k\right\Vert }=0. \notag \\
\end{eqnarray}

On the other hand, taking $\varepsilon \rightarrow 0$ and $k\rightarrow 0$ on both sides of $\rm{Eq.(\ref{X3})}$, we will show that
\begin{equation}\label{X6}
\underset{\varepsilon \rightarrow 0}{\lim }\frac{\left\Vert \rho \left(
a_{1}\,_{i}H_{\gamma ,\beta ,p}^{\rho ,\delta ,q}\left( \varepsilon
_{1}a_{1}^{-\alpha }\right) ,...,a_{n}\,_{i}H_{\gamma ,\beta ,p}^{\rho
,\delta ,q}\left( \varepsilon _{n}a_{n}^{-\alpha }\right) \right)
\right\Vert }{\left\Vert \varepsilon \right\Vert }=0.
\end{equation}

Now, let
\begin{eqnarray}\label{X7}
&&\rho \left( a_{1}\,_{i}H_{\gamma ,\beta ,p}^{\rho ,\delta ,q}\left(
\varepsilon _{1}a_{1}^{-\alpha }\right) ,...,a_{n}\,_{i}H_{\gamma ,\beta
,p}^{\rho ,\delta ,q}\left( \varepsilon _{n}a_{n}^{-\alpha }\right) \right) 
\notag \\
&=&g\left( f\left( a_{1}\,_{i}H_{\gamma ,\beta ,p}^{\rho ,\delta ,q}\left(
\varepsilon _{1}a_{1}^{-\alpha }\right) ,...,a_{n}\,_{i}H_{\gamma ,\beta
,p}^{\rho ,\delta ,q}\left( \varepsilon _{n}a_{n}^{-\alpha }\right) \right)
\right) -g\left( f(a)\right) -M\circ L\left( \varepsilon \right)   \notag \\
&=&g\left( 
\begin{array}{c}
f_{1}\left( a_{1}\,_{i}H_{\gamma ,\beta ,p}^{\rho ,\delta ,q}\left(
\varepsilon _{1}a_{1}^{-\alpha }\right) ,...,a_{n}\,_{i}H_{\gamma ,\beta
,p}^{\rho ,\delta ,q}\left( \varepsilon _{n}a_{n}^{-\alpha }\right) \right)
,... \\ 
...,f_{m}\left( a_{1}\,_{i}H_{\gamma ,\beta ,p}^{\rho ,\delta ,q}\left(
\varepsilon _{1}a_{1}^{-\alpha }\right) ,...,a_{n}\,_{i}H_{\gamma ,\beta
,p}^{\rho ,\delta ,q}\left( \varepsilon _{n}a_{n}^{-\alpha }\right) \right) 
\end{array}%
\right) -g\left( f(a)\right)   \notag \\
&&-M\left( 
\begin{array}{c}
f\left( a_{1}\,_{i}H_{\gamma ,\beta ,p}^{\rho ,\delta ,q}\left( \varepsilon
_{1}a_{1}^{-\alpha }\right) ,...,a_{n}\,_{i}H_{\gamma ,\beta ,p}^{\rho
,\delta ,q}\left( \varepsilon _{n}a_{n}^{-\alpha }\right) \right)  \\ 
-f\left( a\right) -\varphi \left( a_{1}\,_{i}H_{\gamma ,\beta ,p}^{\rho
,\delta ,q}\left( \varepsilon _{1}a_{1}^{-\alpha }\right)
,...,a_{n}\,_{i}H_{\gamma ,\beta ,p}^{\rho ,\delta ,q}\left( \varepsilon
_{n}a_{n}^{-\alpha }\right) \right) 
\end{array}%
\right)   \notag \\
&=&\left[ g\left( 
\begin{array}{c}
f_{1}\left( a_{1}\,_{i}H_{\gamma ,\beta ,p}^{\rho ,\delta ,q}\left(
\varepsilon _{1}a_{1}^{-\alpha }\right) ,...,a_{n}\,_{i}H_{\gamma ,\beta
,p}^{\rho ,\delta ,q}\left( \varepsilon _{n}a_{n}^{-\alpha }\right) \right)
,... \\ 
...,f_{m}\left( a_{1}\,_{i}H_{\gamma ,\beta ,p}^{\rho ,\delta ,q}\left(
\varepsilon _{1}a_{1}^{-\alpha }\right) ,...,a_{n}\,_{i}H_{\gamma ,\beta
,p}^{\rho ,\delta ,q}\left( \varepsilon _{n}a_{n}^{-\alpha }\right) \right) 
\end{array}%
\right) -g\left( f(a)\right) \right.   \notag \\
&&\left. -M\left( 
\begin{array}{c}
f_{1}\left( a_{1}\,_{i}H_{\gamma ,\beta ,p}^{\rho ,\delta ,q}\left(
\varepsilon _{1}a_{1}^{-\alpha }\right) ,...,a_{n}\,_{i}H_{\gamma ,\beta
,p}^{\rho ,\delta ,q}\left( \varepsilon _{n}a_{n}^{-\alpha }\right) \right)
-f_{1}\left( a\right) ,... \\ 
...,f_{m}\left( a_{1}\,_{i}H_{\gamma ,\beta ,p}^{\rho ,\delta ,q}\left(
\varepsilon _{1}a_{1}^{-\alpha }\right) ,...,a_{n}\,_{i}H_{\gamma ,\beta
,p}^{\rho ,\delta ,q}\left( \varepsilon _{n}a_{n}^{-\alpha }\right) \right)
-f_{m}\left( a\right) 
\end{array}%
\right) \right]   \notag \\
&&+M\left[ \varphi \left( a_{1}\,_{i}H_{\gamma ,\beta ,p}^{\rho ,\delta
,q}\left( \varepsilon _{1}a_{1}^{-\alpha }\right) ,...,a_{n}\,_{i}H_{\gamma
,\beta ,p}^{\rho ,\delta ,q}\left( \varepsilon _{n}a_{n}^{-\alpha }\right)
\right) \right] .
\end{eqnarray}

If we put $u_{j}=f_{j}\left( a_{1}\,_{i}H_{\gamma ,\beta ,p}^{\rho ,\delta
,q}\left( \varepsilon _{1}a_{1}^{-\alpha }\right) ,...,a_{n}\,_{i}H_{\gamma
,\beta ,p}^{\rho ,\delta ,q}\left( \varepsilon _{n}a_{n}^{-\alpha }\right)
\right) -f_{j}\left( a\right) $, with $j=1,2,...,m$, then we have
$f_{j}\left( a_{1}\,_{i}H_{\gamma ,\beta ,p}^{\rho ,\delta
,q}\left( \varepsilon _{1}a_{1}^{-\alpha }\right) ,...,a_{n}\,_{i}H_{\gamma
,\beta ,p}^{\rho ,\delta ,q}\left( \varepsilon _{n}a_{n}^{-\alpha }\right)
\right) =u_{j}+f_{j}\left( a\right), $ and $u\rightarrow 0$ as $\varepsilon\rightarrow 0$. Hence, using {\rm{Eq.(\ref{X3})}}, we have
\begin{eqnarray*}
&&\rho \left( a_{1}\,_{i}H_{\gamma ,\beta ,p}^{\rho ,\delta ,q}\left(
\varepsilon _{1}a_{1}^{-\alpha }\right) ,...,a_{n}\,_{i}H_{\gamma ,\beta
,p}^{\rho ,\delta ,q}\left( \varepsilon _{n}a_{n}^{-\alpha }\right) \right) 
\notag \\
&=&\left[ g\left( 
\begin{array}{c}
f_{1}\left( a_{1}\,_{i}H_{\gamma ,\beta ,p}^{\rho ,\delta ,q}\left(
\varepsilon _{1}a_{1}^{-\alpha }\right) ,...,a_{n}\,_{i}H_{\gamma ,\beta
,p}^{\rho ,\delta ,q}\left( \varepsilon _{n}a_{n}^{-\alpha }\right) \right)
,... \\ 
...,f_{m}\left( a_{1}\,_{i}H_{\gamma ,\beta ,p}^{\rho ,\delta ,q}\left(
\varepsilon _{1}a_{1}^{-\alpha }\right) ,...,a_{n}\,_{i}H_{\gamma ,\beta
,p}^{\rho ,\delta ,q}\left( \varepsilon _{n}a_{n}^{-\alpha }\right) \right) 
\end{array}
\right) -g\left( f(a)\right) -M\left( u\right) \right]   \notag \\
&&+M\left[ \varphi \left( a_{1}\,_{i}H_{\gamma ,\beta ,p}^{\rho ,\delta
,q}\left( \varepsilon _{1}a_{1}^{-\alpha }\right) ,...,a_{n}\,_{i}H_{\gamma
,\beta ,p}^{\rho ,\delta ,q}\left( \varepsilon _{n}a_{n}^{-\alpha }\right)
\right) \right]   \notag \\
&=&\psi \left( f_{1}\left( a\right) \,_{i}H_{\gamma ,\beta ,p}^{\rho ,\delta
,q}\left( u_{1}f\left( a\right) ^{-\alpha }\right) ,...,f_{m}\left( a\right)
\,_{i}H_{\gamma ,\beta ,p}^{\rho ,\delta ,q}\left( u_{m}f_{m}\left( a\right)
^{-\alpha }\right) \right)   \notag \\
&&+M\left[ \varphi \left( a_{1}\,_{i}H_{\gamma ,\beta ,p}^{\rho ,\delta
,q}\left( \varepsilon _{1}a_{1}^{-\alpha }\right) ,...,a_{n}\,_{i}H_{\gamma
,\beta ,p}^{\rho ,\delta ,q}\left( \varepsilon _{n}a_{n}^{-\alpha }\right)
\right) \right] .
\end{eqnarray*}

Thus we will show,
\begin{equation}\label{X8}
\underset{u\rightarrow 0}{\lim }\frac{\left\Vert \psi \left( f_{1}\left(
a\right) \,_{i}H_{\gamma ,\beta ,p}^{\rho ,\delta ,q}\left( u_{1}f\left(
a\right) ^{-\alpha }\right) ,...,f_{m}\left( a\right) \,_{i}H_{\gamma ,\beta
,p}^{\rho ,\delta ,q}\left( u_{m}f_{m}\left( a\right) ^{-\alpha }\right)
\right) \right\Vert }{\left\Vert u\right\Vert }=0
\end{equation}
and
\begin{equation}
\underset{\varepsilon \rightarrow 0}{\lim }\frac{\left\Vert M\left( \varphi
\left( a_{1}\,_{i}H_{\gamma ,\beta ,p}^{\rho ,\delta ,q}\left( \varepsilon
_{1}a_{1}^{-\alpha }\right) ,...,a_{n}\,_{i}H_{\gamma ,\beta ,p}^{\rho
,\delta ,q}\left( \varepsilon _{n}a_{n}^{-\alpha }\right) \right) \right)
\right\Vert }{\left\Vert \varepsilon \right\Vert }=0.  \label{X9}
\end{equation}

For {\rm{Eq.(\ref{X8})}}, it is obvious from of {\rm{Eq.(\ref{X5})}}. Now, for  {\rm{Eq.(\ref{X9})}}, we have
\begin{eqnarray}\label{X10}
&&\left\Vert M\left( \varphi \left( a_{1}\,_{i}H_{\gamma ,\beta ,p}^{\rho
,\delta ,q}\left( \varepsilon _{1}a_{1}^{-\alpha }\right)
,...,a_{n}\,_{i}H_{\gamma ,\beta ,p}^{\rho ,\delta ,q}\left( \varepsilon
_{n}a_{n}^{-\alpha }\right) \right) \right) \right\Vert   \notag \\
&\leq &\left\Vert M\right\Vert \left\Vert \left( \varphi \left(
a_{1}\,_{i}H_{\gamma ,\beta ,p}^{\rho ,\delta ,q}\left( \varepsilon
_{1}a_{1}^{-\alpha }\right) ,...,a_{n}\,_{i}H_{\gamma ,\beta ,p}^{\rho
,\delta ,q}\left( \varepsilon _{n}a_{n}^{-\alpha }\right) \right) \right)
\right\Vert   \notag \\
&\leq &K\left\Vert \left( \varphi \left( a_{1}\,_{i}H_{\gamma ,\beta
,p}^{\rho ,\delta ,q}\left( \varepsilon _{1}a_{1}^{-\alpha }\right)
,...,a_{n}\,_{i}H_{\gamma ,\beta ,p}^{\rho ,\delta ,q}\left( \varepsilon
_{n}a_{n}^{-\alpha }\right) \right) \right) \right\Vert,
\end{eqnarray}
such that $K>0$. Taking the limit $\varepsilon\rightarrow 0$ on both sides of {\rm{Eq.(\ref{X10})}} and using {\rm{Eq.(\ref{X4})}}, we get {\rm{Eq.(\ref{X8})}}. Hence, we conclude the proof.
\end{proof}

\begin{corollary}\label{AT} For $m=n=p=1$, the {\rm{Theorem (\ref{AR})}} states that
\begin{equation*}
_{i}^{\rho }\mathcal{V}_{\gamma ,\beta ,\alpha }^{\delta ,p,q}\left( g\circ f\right)
\left( a\right) =g^{\prime }\left( f\left( a\right) \right)\, _{i}^{\rho
}\mathcal{V}_{\gamma ,\beta ,\alpha }^{\delta ,p,q}f\left( a\right) .
\end{equation*}

\end{corollary}

{\rm{Corollary (\ref{AT})}} says that {\rm{Theorem (\ref{AR})}} generalizes {\rm{Theorem (\ref{teo2})}}.

\begin{corollary}\label{AY} Consider all the conditions of {\rm{Theorem (\ref{AR})}} satisfied. Then
\begin{eqnarray*}
&&_{i}^{\rho }\mathbb{V}_{\gamma ,\beta ,\alpha }^{\delta ,p,q}\left( g\circ
f\right) \left( a\right)   \notag \\
&=&g^{\prime }\left( f\left( a\right) \right) \left( 
\tiny\begin{array}{cccc}
\frac{\Gamma \left( \beta \right) \left( \rho \right) _{q}}{\Gamma \left(
\gamma +\beta \right) \left( \delta \right) _{p}}f_{1}\left( a\right)
^{1-\alpha } & 0 & ... & 0 \\ 
0 & \frac{\Gamma \left( \beta \right) \left( \rho \right) _{q}}{\Gamma
\left( \gamma +\beta \right) \left( \delta \right) _{p}}f_{2}\left( a\right)
^{1-\alpha } & ... & 0 \\ 
\vdots  & \vdots  & \ddots  & \vdots  \\ 
0 & 0 & ... & \frac{\Gamma \left( \beta \right) \left( \rho \right) _{q}}{%
\Gamma \left( \gamma +\beta \right) \left( \delta \right) _{p}}f_{n}\left(
a\right) ^{1-\alpha }%
\end{array}
\right) 
\end{eqnarray*}
where,
$\left( 
\tiny\begin{array}{cccc}
\frac{\Gamma \left( \beta \right) \left( \rho \right) _{q}}{\Gamma \left(
\gamma +\beta \right) \left( \delta \right) _{p}}f_{1}\left( a\right)
^{1-\alpha } & 0 & ... & 0 \\ 
0 & \frac{\Gamma \left( \beta \right) \left( \rho \right) _{q}}{\Gamma
\left( \gamma +\beta \right) \left( \delta \right) _{p}}f_{2}\left( a\right)
^{1-\alpha } & ... & 0 \\ 
\vdots  & \vdots  & \ddots  & \vdots  \\ 
0 & 0 & ... & \frac{\Gamma \left( \beta \right) \left( \rho \right) _{q}}{%
\Gamma \left( \gamma +\beta \right) \left( \delta \right) _{p}}f_{n}\left(
a\right) ^{1-\alpha }%
\end{array}%
\right) $, is the matrix corresponding to the linear transformation $_{i}^{\rho }\mathbb{V}_{\gamma ,\beta ,\alpha }^{\delta ,p,q}f\left( a\right)$.

\end{corollary}

\begin{corollary}\label{AC} Consider all the conditions of {\rm{Theorem (\ref{AR})}} satisfied. For $f(a)=a$, {\rm{Corollary (\ref{AY})}}, says that
\begin{eqnarray*}
_{i}^{\rho }V_{\gamma ,\beta ,\alpha }^{\delta ,p,q}g\left( a\right) 
&=&g^{\prime }\left( a\right) \left( 
\tiny\begin{array}{cccc}
\frac{\Gamma \left( \beta \right) \left( \rho \right) _{q}}{\Gamma \left(
\gamma +\beta \right) \left( \delta \right) _{p}}a_{1}^{1-\alpha } & 0 & ...
& 0 \\ 
0 & \frac{\Gamma \left( \beta \right) \left( \rho \right) _{q}}{\Gamma
\left( \gamma +\beta \right) \left( \delta \right) _{p}}a_{2}^{1-\alpha } & 
... & 0 \\ 
\vdots  & \vdots  & \ddots  & \vdots  \\ 
0 & 0 & ... & \frac{\Gamma \left( \beta \right) \left( \rho \right) _{q}}{%
\Gamma \left( \gamma +\beta \right) \left( \delta \right) _{p}}%
a_{n}^{1-\alpha }%
\end{array}%
\right)   \notag \\
&=&g^{\prime }\left( a\right) \frac{\Gamma \left( \beta \right) \left( \rho
\right) _{q}}{\Gamma \left( \gamma +\beta \right) \left( \delta \right) _{p}}%
L_{\alpha }^{1-\alpha }.
\end{eqnarray*}
\end{corollary}

\begin{remark} The {\rm{Corollary (\ref{AC})}} generalizes part 5 of the {\rm{Theorem (\ref{teo2})}}.
\end{remark}

\begin{theorem}\label{JOS} Let $f$ be a vector valued function with $n$ variables such that $f(x_{1},...,x_{n})=(f_{1}(x_{1},...,x_{n}),...,f_{n}(x_{1},...,x_{n}))$. Then $f$ is $\alpha$-differentiable function at $a=(a_{1},...,a_{n})\in\mathbb{R}^n$, with $a_{i}>0$ if, and only if, each $f_{i}$ is,
\begin{equation}
_{i}^{\rho }\mathbb{V}_{\gamma ,\beta ,\alpha }^{\delta ,p,q}f\left(
a\right) =\left( _{i}^{\rho }\mathbb{V}_{\gamma ,\beta ,\alpha }^{\delta
,p,q}f_{1}\left( a\right) ,...,\,_{i}^{\rho }\mathbb{V}_{\gamma ,\beta ,\alpha
}^{\delta ,p,q}f_{m}\left( a\right) \right) ,
\end{equation}
where $\alpha\in (0,1]$ and $\rho,\delta,\gamma,\beta\in\mathbb{C}$, $p,q>0$ with, $Re(\rho)>0$, $Re(\delta)>0$, $Re(\gamma)>0$, $Re(\beta)>0$ and $Re(\gamma)+p\geq q$.
\end{theorem}
\begin{proof}
If  each $f_{i}$ is $\alpha$-differentiable at $a$ and $L=\left( _{i}^{\rho }\mathbb{V}_{\gamma ,\beta ,\alpha }^{\delta,p,q}f_{1}\left( a\right),...,\,_{i}^{\rho }\mathbb{V}_{\gamma ,\beta ,\alpha}^{\delta ,p,q}f_{m}\left( a\right) \right) $, then
\begin{eqnarray}\label{J3}
&&f\left( a_{1}\,_{i}H_{\gamma ,\beta ,p}^{\rho ,\delta ,q}\left( \varepsilon
_{1}a_{1}^{-\alpha }\right) ,...,a_{n}\,_{i}H_{\gamma ,\beta ,p}^{\rho ,\delta
,q}\left( \varepsilon _{n}a_{n}^{-\alpha }\right) \right) -f\left( a\right)
-L\left( \varepsilon \right)   \notag \\
&=&\left[ f_{1}\left( a_{1}\,_{i}H_{\gamma ,\beta ,p}^{\rho ,\delta ,q}\left(
\varepsilon _{1}a_{1}^{-\alpha }\right) ,...,a_{n}\,_{i}H_{\gamma ,\beta
,p}^{\rho ,\delta ,q}\left( \varepsilon _{n}a_{n}^{-\alpha }\right) \right)
-f_{1}\left( a\right) -\,_{i}^{\rho }\mathbb{V}_{\gamma ,\beta ,\alpha
}^{\delta ,p,q}f_{1}\left( a\right) \left( \varepsilon \right) ,...\right.  
\notag \\
&&\left. ...,f_{m}\left( a_{1}\,_{i}H_{\gamma ,\beta ,p}^{\rho ,\delta ,q}\left(
\varepsilon _{1}a_{1}^{-\alpha }\right) ,...,a_{n}\,_{i}H_{\gamma ,\beta
,p}^{\rho ,\delta ,q}\left( \varepsilon _{n}a_{n}^{-\alpha }\right) \right)
-f_{m}\left( a\right) -\,_{i}^{\rho }\mathbb{V}_{\gamma ,\beta ,\alpha
}^{\delta ,p,q}f_{m}\left( a\right) \left( \varepsilon \right) \right]. \notag \\
\end{eqnarray}

Taking the limit $\varepsilon\rightarrow 0$ on both sides of {\rm Eq.(\ref{J3})}, we have
\begin{eqnarray*}
&&\underset{\varepsilon \rightarrow 0}{\lim }\frac{\left\Vert f\left(
a_{1}\,_{i}H_{\gamma ,\beta ,p}^{\rho ,\delta ,q}\left( \varepsilon
_{1}a_{1}^{-\alpha }\right) ,...,a_{n}\,_{i}H_{\gamma ,\beta ,p}^{\rho ,\delta
,q}\left( \varepsilon _{n}a_{n}^{-\alpha }\right) \right) -f\left( a\right)
-L\left( \varepsilon \right) \right\Vert }{\left\Vert \varepsilon
\right\Vert }  \notag \\
&=&\underset{\varepsilon \rightarrow 0}{\lim }\frac{\left\Vert \overset{n}{%
\underset{j=1}{\sum }}f_{j}\left( a_{j}\,_{i}H_{\gamma ,\beta ,p}^{\rho ,\delta
,q}\left( \varepsilon _{j}a_{j}^{-\alpha }\right) \right) -\,f_{j}\left(
a\right) -\,_{i}^{\rho }\mathbb{V}_{\gamma ,\beta ,\alpha }^{\delta
,p,q}f_{j}\left( a\right) \left( \varepsilon \right) \right\Vert }{%
\left\Vert \varepsilon \right\Vert }  \notag \\
&\leq &\underset{\varepsilon \rightarrow 0}{\lim }\overset{n}{\underset{j=1}{%
\sum }}\frac{\left\Vert f_{j}\left( a_{j}\,_{i}H_{\gamma ,\beta ,p}^{\rho ,\delta
,q}\left( \varepsilon _{j}a_{j}^{-\alpha }\right) \right) -\,f_{j}\left(
a\right) -\,_{i}^{\rho }\mathbb{V}_{\gamma ,\beta ,\alpha }^{\delta
,p,q}f_{j}\left( a\right) \left( \varepsilon \right) \right\Vert }{%
\left\Vert \varepsilon \right\Vert }=0,
\end{eqnarray*}
which is the result.
\end{proof}

\begin{theorem} Let $0<\alpha\leq 1$, $\lambda,\mu\in\mathbb{R}$, $\gamma ,\beta ,\rho ,\delta \in \mathbb{C}$ and $p,q>0$ such that ${Re}\left( \gamma \right) >0$, ${Re}\left( \beta \right) >0$, ${Re}\left( \rho \right) >0$, ${Re}\left( \delta \right) >0$, ${Re}\left( \gamma \right) +p\geq q$ and $f,g$ $\alpha$-differentiable at $a=(a_{1},...,a_{n})\in\mathbb{R}^n$, with $a_{i}>0$. Then,

\begin{enumerate}
\item $ _{i}^{\rho }\mathbb{V}_{\gamma ,\beta ,\alpha }^{\delta ,p,q}\left( \lambda
f+\mu g\right) \left( a\right) =\lambda _{i}^{\rho }\mathbb{V}_{\gamma,\beta ,\alpha }^{\delta ,p,q}f\left( a\right) +\mu _{i}^{\rho }\mathbb{V}%
_{\gamma ,\beta ,\alpha }^{\delta ,p,q}g\left( a\right) $.

\item $_{i}^{\rho }\mathbb{V}_{\gamma ,\beta ,\alpha }^{\delta ,p,q}\left( f\cdot
g\right) \left( a\right) =f\left( a\right) _{i}^{\rho }\mathbb{V}_{\gamma
,\beta ,\alpha }^{\delta ,p,q}g\left( a\right) +g\left( a\right) _{i}^{\rho }%
\mathbb{V}_{\gamma ,\beta ,\alpha }^{\delta ,p,q}f\left( a\right) $.
\end{enumerate}
\end{theorem}

\begin{proof} 
{\rm 1.} Let $A=a_{1}\,_{i}H_{\gamma ,\beta ,p}^{\rho ,\delta ,q}\left( \varepsilon _{1}a_{1}^{-\alpha }\right) +\cdot \cdot \cdot +a_{n}\,_{i}H_{\gamma ,\beta ,p}^{\rho ,\delta ,q}\left( \varepsilon _{1}a_{n}^{-\alpha }\right) $, then we have,
\begin{eqnarray*}
&&\underset{\varepsilon \rightarrow 0}{\lim }\frac{\left\Vert \left( \lambda
f+\mu g\right) \left( A\right) -\left( \lambda f+\mu g\right) \left(
a\right) -\left( \lambda \,_{i}^{\rho }\mathbb{V}_{\gamma ,\beta ,\alpha
}^{\delta ,p,q}f\left( a\right) +\mu \,_{i}^{\rho }\mathbb{V}_{\gamma ,\beta
,\alpha }^{\delta ,p,q}g\left( a\right) \right) \left( \varepsilon \right)
\right\Vert }{\left\Vert \varepsilon \right\Vert }  \notag \\
&=&\underset{\varepsilon \rightarrow 0}{\lim }\frac{\left\Vert \lambda
f\left( A\right) -\lambda f\left( a\right) -\lambda _{i}^{\rho }\mathbb{V}%
_{\gamma ,\beta ,\alpha }^{\delta ,p,q}f\left( a\right) \left( \varepsilon
\right) +\mu g\left( A\right) -\mu g\left( a\right) -\mu \,_{i}^{\rho }\mathbb{%
V}_{\gamma ,\beta ,\alpha }^{\delta ,p,q}g\left( a\right) \left( \varepsilon
\right) \right\Vert }{\left\Vert \varepsilon \right\Vert }  \notag \\
&\leq &\underset{\varepsilon \rightarrow 0}{\lim }\frac{\left\Vert \lambda
f\left( A\right) -\lambda f\left( a\right) -\lambda _{i}^{\rho }\mathbb{V}%
_{\gamma ,\beta ,\alpha }^{\delta ,p,q}f\left( a\right) \left( \varepsilon
\right) \right\Vert }{\left\Vert \varepsilon \right\Vert }+\underset{%
\varepsilon \rightarrow 0}{\lim }\frac{\left\Vert \mu g\left( A\right) -\mu
g\left( a\right) -\mu\, _{i}^{\rho }\mathbb{V}_{\gamma ,\beta ,\alpha
}^{\delta ,p,q}g\left( a\right) \left( \varepsilon \right) \right\Vert }{%
\left\Vert \varepsilon \right\Vert }  \notag \\
&=&\lambda \underset{\varepsilon \rightarrow 0}{\lim }\frac{\left\Vert
f\left( A\right) -f\left( a\right) -\, _{i}^{\rho }\mathbb{V}_{\gamma ,\beta
,\alpha }^{\delta ,p,q}f\left( a\right) \left( \varepsilon \right)
\right\Vert }{\left\Vert \varepsilon \right\Vert }+\mu \underset{%
\varepsilon \rightarrow 0}{\lim }\frac{\left\Vert g\left( A\right) -g\left(
a\right) -\,_{i}^{\rho }\mathbb{V}_{\gamma ,\beta ,\alpha }^{\delta
,p,q}g\left( a\right) \left( \varepsilon \right) \right\Vert }{\left\Vert
\varepsilon \right\Vert }=0.
\end{eqnarray*}

So, the proof is complete.

{\rm 2.} Let $A=a_{1}\,_{i}H_{\gamma ,\beta ,p}^{\rho ,\delta ,q}\left( \varepsilon _{1}a_{1}^{-\alpha }\right) +\cdot \cdot \cdot +a_{n}\,_{i}H_{\gamma ,\beta ,p}^{\rho ,\delta ,q}\left( \varepsilon _{1}a_{n}^{-\alpha }\right) $, then we have,
\begin{eqnarray*}
&&\underset{\varepsilon \rightarrow 0}{\lim }\frac{\left\Vert \left( f\cdot
g\right) \left( A\right) -\left( f\cdot g\right) \left( a\right) -\left(
f\left( a\right)\, _{i}^{\rho }\mathbb{V}_{\gamma ,\beta ,\alpha }^{\delta
,p,q}g\left( a\right) +g\left( a\right) _{i}^{\rho }\mathbb{V}_{\gamma
,\beta ,\alpha }^{\delta ,p,q}f\left( a\right) \right) \left( \varepsilon
\right) \right\Vert }{\left\Vert \varepsilon \right\Vert }  \notag \\
&=&\underset{\varepsilon \rightarrow 0}{\lim }\frac{\left\Vert 
\begin{array}{c}
f\left( A\right) g\left( A\right) -f\left( a\right) g\left( A\right)
-g\left( A\right) _{i}^{\rho }\mathbb{V}_{\gamma ,\beta ,\alpha }^{\delta
,p,q}f\left( a\right) \left( \varepsilon \right) +f\left( a\right) g\left(
A\right) -f\left( a\right) g\left( a\right)  \\ 
-f\left( a\right) _{i}^{\rho }\mathbb{V}_{\gamma ,\beta ,\alpha }^{\delta
,p,q}g\left( a\right) \left( \varepsilon \right) +g\left( A\right)
_{i}^{\rho }\mathbb{V}_{\gamma ,\beta ,\alpha }^{\delta ,p,q}f\left(
a\right) \left( \varepsilon \right) -g\left( a\right) _{i}^{\rho }\mathbb{V}%
_{\gamma ,\beta ,\alpha }^{\delta ,p,q}f\left( a\right) \left( \varepsilon
\right) 
\end{array}%
\right\Vert }{\left\Vert \varepsilon \right\Vert }  \notag \\
&\leq &\underset{\varepsilon \rightarrow 0}{\lim }\frac{\left\Vert f\left(
A\right) g\left( A\right) -f\left( a\right) g\left( A\right) -g\left(
A\right) _{i}^{\rho }\mathbb{V}_{\gamma ,\beta ,\alpha }^{\delta
,p,q}f\left( a\right) \left( \varepsilon \right) \right\Vert }{\left\Vert
\varepsilon \right\Vert }+  \notag \\
&&\underset{\varepsilon \rightarrow 0}{\lim }\frac{\left\Vert f\left(
a\right) g\left( A\right) -f\left( a\right) g\left( a\right) -f\left(
a\right) _{i}^{\rho }\mathbb{V}_{\gamma ,\beta ,\alpha }^{\delta
,p,q}g\left( a\right) \left( \varepsilon \right) \right\Vert }{\left\Vert
\varepsilon \right\Vert }+  \notag \\
&&\underset{\varepsilon \rightarrow 0}{\lim }\frac{\left\Vert g\left(
A\right) _{i}^{\rho }\mathbb{V}_{\gamma ,\beta ,\alpha }^{\delta
,p,q}f\left( a\right) \left( \varepsilon \right) -g\left( a\right)
_{i}^{\rho }\mathbb{V}_{\gamma ,\beta ,\alpha }^{\delta ,p,q}f\left(
a\right) \left( \varepsilon \right) \right\Vert }{\left\Vert \varepsilon
\right\Vert }  \notag \\
&=&\underset{\varepsilon \rightarrow 0}{\lim }\left\Vert _{i}^{\rho }\mathbb{%
V}_{\gamma ,\beta ,\alpha }^{\delta ,p,q}f\left( a\right) \left( \varepsilon
\right) \right\Vert \frac{\left\Vert g\left( A\right) -g\left( a\right)
\right\Vert }{\left\Vert \varepsilon \right\Vert }  \notag \\
&\leq &K\underset{\varepsilon \rightarrow 0}{\lim }\left\Vert \left(
\varepsilon \right) \right\Vert \frac{\left\Vert g\left( A\right) -g\left(
a\right) \right\Vert }{\left\Vert \varepsilon \right\Vert }=0,
\end{eqnarray*}
with $K>0$. So, the proof is complete.
\end{proof}

\section{ Truncated $\mathcal{V}$-fractional partial derivatives and applications}
In this section, we introduce the truncated $\mathcal{V}$-fractional partial derivative and discuss applications: the theorem associated with the commutativity property of two truncated $\mathcal{V}$-fractional partial derivatives and the truncated $\mathcal{V}$-fractional Green's theorem.

\begin{definition}Let $f$ be a real valued function with $n$ variables and $a=(a_{1},...,a_{n})\in\mathbb{R}^n$ be a point whose $i^{th}$ component is positive. Then, the limit

\begin{equation}
\underset{\varepsilon \rightarrow 0}{\lim }\frac{f\left(
a_{1},...,a_{j}\,_{i}H_{\gamma ,\beta ,p}^{\rho ,\delta ,q}\left( \varepsilon
a_{j}^{-\alpha }\right) ,...,a_{n}\right) -f\left( a_{1},...,a_{n}\right) }{\varepsilon },
\end{equation}
if it exists, is denoted by $\displaystyle\frac{\partial ^{\alpha }}{\partial x^{\alpha }}f\left( a\right) :=\displaystyle\frac{\partial ^{\alpha }}{\partial x^{\alpha }}f\left( x\right) \mid _{x=a}$, and called the $i^{th}$ truncated $\mathcal{V}$-fractional partial derivative of $f$ of order $\alpha\in(0,1]$ at $a$.
\end{definition}

\begin{theorem} Let $f$ be a vector valued function with $n$ variables. If $f$ is $\alpha$-differentiable at $a=(a_{1},...,a_{n})\in\mathbb{R}^n$, with $a_{j}>0$, then $\displaystyle\frac{\partial ^{\alpha }}{\partial x^{\alpha }_{p}}f_{j}\left( a\right) $ of order $\alpha\in(0,1]$ exists for $1\leq j \leq m$,  $1\leq p \leq n$ and the Jacobian of $f$ at $a$ is the $m\times n$ matrix $\left( \displaystyle\frac{\partial ^{\alpha }}{\partial x_{p}^{\alpha }}f_{j}\left(a\right) \right) $.
\end{theorem}
\begin{proof} Let $f\left( x_{1},...,x_{n}\right) =\left( f_{1}\left( x_{1},...,x_{n}\right)
,...,f_{m}\left( x_{1},...,x_{n}\right) \right) $. Suppose first that $m=1$, so that $f\left( x_{1},...,x_{n}\right) \in \mathbb{R}^{n}$. Define $h:\mathbb{R}\rightarrow\mathbb{R}^n$ by $h\left( y\right) =\left( a_{1},...,y,...,a_{n}\right) $ with $y$ in the place of $p^{th}$. Then $\displaystyle\frac{\partial ^{\alpha }}{\partial x_{p}^{\alpha }}f_{j}\left( a\right) =\,_{i}^{\rho }\mathbb{V}_{\gamma ,\beta ,\alpha }^{\delta ,p,q}\left( f\circ h\right)\left( a_{p}\right) $. Hence, by {\rm{Corollary (\ref{AY})}}, we have
\begin{eqnarray*}
&&_{i}^{\rho }\mathbb{V}_{\gamma ,\beta ,\alpha }^{\delta ,p,q}\left( f\circ
h\right) \left( a_{p}\right)   \notag \\
&=&f^{\prime }\left( h\left( a_{p}\right) \right) \left( 
\tiny\begin{array}{cccc}
\frac{\Gamma \left( \beta \right) \left( \rho \right) _{q}}{\Gamma \left(
\gamma +\beta \right) \left( \delta \right) _{p}}h_{1}\left( a_{p}\right)
^{1-\alpha } & 0 & ... & 0 \\ 
0 & \frac{\Gamma \left( \beta \right) \left( \rho \right) _{q}}{\Gamma
\left( \gamma +\beta \right) \left( \delta \right) _{p}}h_{2}\left(
a_{p}\right) ^{1-\alpha } & ... & 0 \\ 
\vdots  & \vdots  & \ddots  & \vdots  \\ 
0 & 0 & ... & \frac{\Gamma \left( \beta \right) \left( \rho \right) _{q}}{%
\Gamma \left( \gamma +\beta \right) \left( \delta \right) _{p}}h_{n}\left(
a_{p}\right) ^{1-\alpha }%
\end{array}%
\right)   \notag \\
&=&f^{\prime }\left( a\right) \left( 
\tiny\begin{array}{cccc}
\frac{\Gamma \left( \beta \right) \left( \rho \right) _{q}}{\Gamma \left(
\gamma +\beta \right) \left( \delta \right) _{p}}a_{1}^{1-\alpha } & 0 & ...
& 0 \\ 
0 & \frac{\Gamma \left( \beta \right) \left( \rho \right) _{q}}{\Gamma
\left( \gamma +\beta \right) \left( \delta \right) _{p}}a_{j}^{1-\alpha } & 
... & 0 \\ 
\vdots  & \vdots  & \ddots  & \vdots  \\ 
0 & 0 & ... & \frac{\Gamma \left( \beta \right) \left( \rho \right) _{q}}{%
\Gamma \left( \gamma +\beta \right) \left( \delta \right) _{p}}%
a_{n}^{1-\alpha }%
\end{array}%
\right) =_{i}^{\rho }\mathbb{V}_{\gamma ,\beta ,\alpha }^{\delta
,p,q}f\left( a\right).
\end{eqnarray*}

Since $\left( f\circ h\right)\left( a_{p}\right) $ has a single entry $\displaystyle\frac{\partial ^{\alpha }}{\partial x_{p}^{\alpha }}f_{j}\left( a\right)$, this shows that $\displaystyle\frac{\partial ^{\alpha }}{\partial x_{p}^{\alpha }}f_{j}\left( a\right)$ exists and is the $p^{th}$ entry of the $1\times n$ matrix $\,_{i}^{\rho }\mathbb{V}_{\gamma ,\beta ,\alpha }^{\delta ,p,q}f\left( a\right) $. The theorem now follows for arbitrary $m$ since, by {\rm{Theorem (\ref{JOS})}}, each $f_{j}$, is $\alpha$-differentiable and the $p^{th}$ row of $\,_{i}^{\rho }\mathbb{V}_{\gamma ,\beta ,\alpha }^{\delta ,p,q}f\left( a\right) $ is $\,_{i}^{\rho }\mathbb{V}_{\gamma ,\beta ,\alpha }^{\delta ,p,q}f_{j}\left( a\right) $.
\end{proof}

For the next result, we use the Clairaut-Schwarz theorem integer order \cite{STW}, and realize an application of the truncated $\mathcal{V}$-fractional partial derivative.

\begin{theorem}\label{PO} Assume that $f(t,s)$ is a function for which $\partial _{t}^{\alpha }\left( \partial _{s}^{\kappa }f\left( t,s\right) \right) $ is of order $\alpha\in(0,1]$ and $\partial _{s}^{\kappa }\left( \partial _{t}^{\alpha }f\left( t,s\right)\right) $ is of order $\kappa\in(0,1]$ exist and are continuous over the domain $D\subset \mathbb{R}^{2}$, then
\begin{equation*}
\frac{\partial ^{\alpha }}{\partial t^{\alpha }}\left( \frac{\partial
^{\kappa }}{\partial t^{\kappa }}f\left( t,s\right) \right) =\frac{\partial
^{\kappa }}{\partial t^{\kappa }}\left( \frac{\partial ^{\alpha }}{\partial
t^{\alpha }}f\left( t,s\right) \right) .
\end{equation*}
\end{theorem}
\begin{proof}
By means of the {\rm{Definition (\ref{def7})}}, truncated $\mathcal{V}$-fractional derivative at the $s$ variable, we have
\begin{eqnarray*}
\frac{\partial ^{\alpha }}{\partial t^{\alpha }}\left( \frac{\partial
^{\kappa }}{\partial t^{\kappa }}f\left( t,s\right) \right)  &=&\frac{%
\partial ^{\alpha }}{\partial t^{\alpha }}\left( \underset{\varepsilon
\rightarrow 0}{\lim }\frac{f\left( t,s\,_{i}H_{\gamma ,\beta ,p}^{\rho ,\delta
,q}\left( \varepsilon s^{-\kappa }\right) \right) -f\left( t,s\right) }{%
\varepsilon }\right)   \notag \\
&=&\frac{\partial ^{\alpha }}{\partial t^{\alpha }}\left( \underset{%
\varepsilon \rightarrow 0}{\lim }\frac{f\left( t,s+\displaystyle\frac{\Gamma \left( \beta
\right) \left( \rho \right) _{q}}{\Gamma \left( \gamma +\beta \right) \left(
\delta \right) _{p}}\varepsilon s^{1-\kappa }+O\left( \varepsilon
^{2}\right) \right) -f\left( t,s\right) }{\varepsilon }\right) .
\end{eqnarray*}

Introducing the following change of variable $h=\varepsilon s^{1-\kappa }\left( \displaystyle\frac{\Gamma \left( \beta \right) \left( \rho \right) _{q}}{\Gamma \left( \gamma +\beta \right) \left( \delta \right)
_{p}}+O\left( \varepsilon \right) \right) $ implies $\varepsilon =\displaystyle\frac{h}{s^{1-\kappa }\left( \displaystyle \frac{\Gamma \left( \beta \right) \left( \rho \right) _{q}}{\Gamma \left( \gamma +\beta \right) \left( \delta \right) _{p}}+O\left( \varepsilon \right) \right) }$, we get
\begin{equation*}
\frac{\partial ^{\alpha }}{\partial t^{\alpha }}\left( \frac{\partial
^{\kappa }}{\partial t^{\kappa }}f\left( t,s\right) \right) =\frac{\partial
^{\alpha }}{\partial t^{\alpha }}\left( \underset{\varepsilon \rightarrow 0}{%
\lim }\frac{\displaystyle\frac{f\left( t,s+h\right) -f\left( t,s\right) }{hs^{\kappa -1}}%
}{\displaystyle\frac{\Gamma \left( \beta \right) \left( \rho \right) _{q}}{\Gamma \left(
\gamma +\beta \right) \left( \delta \right) _{p}}+O\left( \varepsilon
\right) }\right).
\end{equation*}

Since $f$ is differentiable in $s$-direction, we obtain
\begin{equation*}
\frac{\partial ^{\alpha }}{\partial t^{\alpha }}\left( \frac{\partial
^{\kappa }}{\partial t^{\kappa }}f\left( t,s\right) \right) =s^{1-\kappa }%
\frac{\Gamma \left( \beta \right) \left( \rho \right) _{q}}{\Gamma \left(
\gamma +\beta \right) \left( \delta \right) _{p}}\frac{\partial ^{\alpha }}{%
\partial t^{\alpha }}\left( \frac{\partial }{\partial s}f\left( t,s\right)\right).
\end{equation*}

Again, by the definition of the truncated $\mathcal{V}$-fractional derivative we have
\begin{equation*}
\frac{\partial ^{\alpha }}{\partial t^{\alpha }}\left( \frac{\partial
^{\kappa }}{\partial t^{\kappa }}f\left( t,s\right) \right) =s^{1-\kappa }%
\frac{\Gamma \left( \beta \right) \left( \rho \right) _{q}}{\Gamma \left(
\gamma +\beta \right) \left( \delta \right) _{p}}\left\{ \underset{%
\varepsilon \rightarrow 0}{\lim }\frac{\frac{\partial }{\partial s}f\left(
t\,_{i}H_{\gamma ,\delta ,p}^{\rho ,\delta ,q}\left( \varepsilon t^{-\alpha
}\right) ,s\right) -\frac{\partial }{\partial s}f\left( t,s\right) }{%
\varepsilon }\right\}.
\end{equation*}

In analogy to the expression, after making a similar change of variable, we have
\begin{equation*}
\frac{\partial ^{\alpha }}{\partial t^{\alpha }}\left( \frac{\partial
^{\kappa }}{\partial t^{\kappa }}f\left( t,s\right) \right) =\frac{\Gamma
\left( \beta \right) \left( \rho \right) _{q}s^{1-\kappa }t^{1-\alpha }}{%
\Gamma \left( \gamma +\beta \right) \left( \delta \right) _{p}}\underset{%
k\rightarrow 0}{\lim }\frac{\frac{\partial }{\partial s}f\left( t+k,s\right)
-\frac{\partial }{\partial s}f\left( t,s\right) }{k}.
\end{equation*}

Since $f$ is differentiable in $t$-direction, we obtain
\begin{equation}\label{LK}
\frac{\partial ^{\alpha }}{\partial t^{\alpha }}\left( \frac{\partial
^{\kappa }}{\partial t^{\kappa }}f\left( t,s\right) \right) =\left( \frac{%
\Gamma \left( \beta \right) \left( \rho \right) _{q}}{\Gamma \left( \gamma
+\beta \right) \left( \delta \right) _{p}}\right) ^{2}s^{1-\kappa
}t^{1-\alpha }\frac{\partial ^{2}}{\partial t\partial s}f\left( t,s\right). 
\end{equation}

Being $f$ a continuous function and using the Clairaut-Schwarz theorem for partial derivative, it follows that
\begin{equation*}
\frac{\partial ^{2}}{\partial t\partial s}f\left( t,s\right) =\frac{\partial
^{2}}{\partial s\partial t}f\left( t,s\right).
\end{equation*}

Therefore the {\rm{Eq.(\ref{LK})}}, becomes
\begin{eqnarray}\label{LK1}
\frac{\partial ^{\alpha }}{\partial t^{\alpha }}\left( \frac{\partial
^{\kappa }}{\partial t^{\kappa }}f\left( t,s\right) \right)  &=&\left( \frac{%
\Gamma \left( \beta \right) \left( \rho \right) _{q}}{\Gamma \left( \gamma
+\beta \right) \left( \delta \right) _{p}}\right) ^{2}s^{1-\kappa
}t^{1-\alpha }\frac{\partial ^{2}}{\partial s\partial t}f\left( t,s\right)  
\notag \\
&=&\left( \frac{\Gamma \left( \beta \right) \left( \rho \right) _{q}}{\Gamma
\left( \gamma +\beta \right) \left( \delta \right) _{p}}\right)
^{2}s^{1-\kappa }t^{1-\alpha }\underset{h\rightarrow 0}{\lim }\frac{\frac{%
\partial }{\partial t}f\left( t,s+h\right) -\frac{\partial }{\partial t}%
f\left( t,s\right) }{h}.\notag \\
\end{eqnarray}

Thus, taking $h=\varepsilon s^{1-\kappa }\left( \displaystyle\frac{\Gamma \left( \beta \right) \left(
\rho \right) _{q}}{\Gamma \left( \gamma +\beta \right) \left( \delta \right)
_{p}}+O\left( \varepsilon \right) \right) $ and later $k=\varepsilon t^{1-\alpha }\left( \displaystyle\frac{\Gamma \left( \beta \right) \left(\rho \right) _{q}}{\Gamma \left( \gamma +\beta \right) \left( \delta \right)
_{p}}+O\left( \varepsilon \right) \right) $ in the {\rm{Eq.(\ref{LK1})}}, we arrive at
\begin{equation*}
\frac{\partial ^{\alpha }}{\partial t^{\alpha }}\left( \frac{\partial
^{\kappa }}{\partial t^{\kappa }}f\left( t,s\right) \right) =\frac{\partial
^{\kappa }}{\partial t^{\kappa }}\left( \underset{h\rightarrow 0}{\lim }%
\frac{\frac{\partial }{\partial t}f\left( t,s+h\right) -\frac{\partial }{%
\partial t}f\left( t,s\right) }{h}\right) =\frac{\partial ^{\kappa }}{%
\partial t^{\kappa }}\left( \frac{\partial ^{\alpha }}{\partial t^{\alpha }}%
f\left( t,s\right) \right),
\end{equation*}
which completes the proof.
\end{proof}

We define the $\mathcal{V}$-fractional vector at the point $a$, given by
\begin{equation*}
\bigtriangledown _{\alpha }f\left( a\right) =\left( \frac{\partial ^{\alpha }%
}{\partial t^{\alpha }}f\left( a\right) ,\frac{\partial ^{\kappa }}{\partial
s^{\kappa }}f\left( a\right) \right) .
\end{equation*}

The next example, is a direct application of the {\rm Theorem \ref{PO}}.

\begin{example} Consider $f(t,s)=e^{a(t+s)}$ with $a\in\mathbb{R}$ satisfying the conditions of {\rm Theorem\ref{PO}}, then we have
\begin{eqnarray}\label{PO1}
\frac{\partial ^{\alpha }}{\partial t^{\alpha }}\left( \frac{\partial
^{\kappa }}{\partial s^{\kappa }}f\left( t,s\right) \right)  &=&\frac{%
\partial ^{\alpha }}{\partial t^{\alpha }}\left( \frac{s^{1-\kappa }\Gamma
\left( \beta \right) \left( \rho \right) _{q}}{\Gamma \left( \gamma +\beta
\right) \left( \delta \right) _{p}}\frac{\partial }{\partial s}e^{a\left(
s+t\right) }\right)   \nonumber \\
&=&a\frac{s^{1-\kappa }\Gamma \left( \beta \right) \left( \rho \right) _{q}}{%
\Gamma \left( \gamma +\beta \right) \left( \delta \right) _{p}}\frac{%
\partial ^{\alpha }}{\partial t^{\alpha }}e^{a\left( s+t\right) }  \nonumber
\\
&=&a^{2}s^{1-\kappa }t^{1-\alpha }\left( \frac{\Gamma \left( \beta \right)
\left( \rho \right) _{q}}{\Gamma \left( \gamma +\beta \right) \left( \delta
\right) _{p}}\right) ^{2}e^{a\left( t+s\right) }
\end{eqnarray}
and
\begin{eqnarray}\label{PO2}
\frac{\partial ^{\kappa }}{\partial s^{\kappa }}\left( \frac{\partial
^{\alpha }}{\partial t^{\alpha }}f\left( t,s\right) \right)  &=&\frac{%
\partial ^{\kappa }}{\partial s^{\kappa }}\left( \frac{t^{1-\alpha }\Gamma
\left( \beta \right) \left( \rho \right) _{q}}{\Gamma \left( \gamma +\beta		
\right) \left( \delta \right) _{p}}\frac{\partial }{\partial t}e^{a\left(
s+t\right) }\right)   \nonumber \\
&=&a\frac{t^{1-\alpha }\Gamma \left( \beta \right) \left( \rho \right) _{q}}{%
\Gamma \left( \gamma +\beta \right) \left( \delta \right) _{p}}\frac{%
\partial ^{\kappa }}{\partial s^{\kappa }}e^{a\left( s+t\right) }  \nonumber
\\
&=&a^{2}s^{1-\kappa }t^{1-\alpha }\left( \frac{\Gamma \left( \beta \right)
\left( \rho \right) _{q}}{\Gamma \left( \gamma +\beta \right) \left( \delta
\right) _{p}}\right) ^{2}e^{a\left( t+s\right) }.
\end{eqnarray}

Thus, by {\rm Eq.(\ref{PO1})} and {\rm Eq.(\ref{PO2})} we conclude that 
\begin{equation*}
\frac{\partial ^{\alpha }}{\partial t^{\alpha }}\left( \frac{\partial
^{\kappa }}{\partial s^{\kappa }}f\left( t,s\right) \right) =\frac{\partial
^{\kappa }}{\partial s^{\kappa }}\left( \frac{\partial ^{\alpha }}{\partial
t^{\alpha }}f\left( t,s\right) \right) .
\end{equation*}
\end{example}

\begin{theorem} {\rm (truncated $\mathcal{V}$-fractional Green’s theorem)} Let $C$ be a simple positively oriented, piecewise smooth and close curve in $\mathbb{R}^2$, say for instance the $x-y$ plane, furthermore assume $D$ in the interior of $C$. If $f(x,y)$ and $g(x,y)$ are two functions having continuous partial truncated $\mathcal{V}$-fractional derivative on $D$ then
\begin{equation*}
\int \int_{D}\left( \frac{\partial ^{\alpha }}{\partial x^{\alpha }}g -\frac{\partial ^{\alpha }}{\partial y^{\alpha }}f \right) d_{\omega }S=\int_{C}\frac{\partial ^{\alpha -1}}{%
\partial y^{\alpha -1}}f d_{\omega }x-\frac{\partial
^{\alpha -1}}{\partial x^{\alpha -1}}g d_{\omega }y,
\end{equation*}
where $d_{\omega }S=\left( \dfrac{\Gamma \left( \gamma +\beta \right) \left( \delta
\right) _{p}}{\Gamma \left( \beta \right) \left( \rho \right) _{q}}\right)
^{2}x^{\alpha -1}y^{\alpha -1}dxdy$, with $d_{\omega}x$ and $d_{\omega}y$, given by {\rm Remark {\ref{fe}}}.
\end{theorem}
\begin{proof}
In fact, note that
\begin{equation}\label{PO4}
\int \int_{D}\left( \frac{\partial ^{\alpha }}{\partial x^{\alpha }}g -\frac{\partial ^{\alpha }}{\partial y^{\alpha }}f\right) d_{\omega }S=\int \int_{D}\left[ \frac{\partial }{%
\partial x}\left( \frac{\partial ^{\alpha -1}}{\partial x^{\alpha -1}}%
g \right) -\frac{\partial }{\partial y}\left( \frac{%
\partial ^{\alpha -1}}{\partial y^{\alpha -1}}f \right) %
\right] d_{\omega }S.
\end{equation}

Applying the classical version of the Green's theorem {\rm \cite{CAC}},
\begin{equation*}
\int \int_{D}\left( \frac{\partial Q}{\partial x}-\frac{\partial }{\partial y
}P\right) dS=\int_{C}\left( Pdx+Qdy\right) 
\end{equation*}
into {\rm Eq.(\ref{PO4})}, we conclude that
\begin{equation*}
\int \int_{D}\left( \frac{\partial ^{\alpha }}{\partial x^{\alpha }}g -\frac{\partial ^{\alpha }}{\partial y^{\alpha }}f \right) d_{\omega }S=\int_{C}\frac{\partial ^{\alpha -1}}{%
\partial y^{\alpha -1}}f d_{\omega }x+\frac{\partial
^{\alpha -1}}{\partial x^{\alpha -1}}g d_{\omega }y.
\end{equation*}
\end{proof}


\section{Concluding remarks}
After a brief introduction to the truncated six-parameters Mittag-Leffler function and the truncated $\mathcal{V}$-fractional derivative with domain of function in $\mathbb{R}$ and the validity of some important results, we have introduced the multivariable truncated $\mathcal{V}$-fractional derivative, that is, with domain of the function in $\mathbb{R}^n$. In this sense, we discussed and proved classical theorems such as: the chain rule, the commutativity of the exponent of two truncated $\mathcal{V}$-fractional derivatives and Green's theorem.

We conclued that: a variety of new fractional derivatives of said local have been recently introduced, all them satisfy the requirements of the integer-order derivative, and have been employed to deal more effectively with real problems and their physical properties \cite{CAC,CAP,MKR,BDMA}. The dynamics of systems over time, becomes more complex and more precise mathematical tools are needed to solve certain theoretical and practical problems. In this theoretical and applicable sense, we extended the idea of truncated $\mathcal{V}$-fractional derivative of a variable, so it is possible to work with differential equations with several variables consequently make comparisons with the results obtained by means of other fractional derivatives. Studies in direction will be published in a forthcoming paper.

\bibliography{ref}
\bibliographystyle{plain}

\end{document}